\documentclass[12pt]{article}

\usepackage[english]{babel}
\usepackage{comment}
\usepackage{setspace}
\usepackage[letterpaper,top=2cm,bottom=2cm,left=3cm,right=3cm,marginparwidth=1.75cm]{geometry}

\usepackage{amsmath}
\usepackage{mathrsfs}
\usepackage{graphicx}
\usepackage{algorithm}
\usepackage{algpseudocode}
\usepackage{subfig}

\usepackage[font={footnotesize,stretch=1.3}]{caption}

\providecommand{\keywords}[1] 
{
  \small	
  \textit{\textbf{Keywords}}---strongly regular graphs, spectral graph theory, spanning tree
}

\usepackage{tikz,tkz-euclide}
\usetikzlibrary{decorations.pathreplacing,angles,quotes}

\usepackage{algpseudocode}
\usepackage{amsthm}
\usepackage{amsfonts}
\usepackage{amssymb}
\usepackage{mathtools}

\usepackage[colorlinks=true, allcolors=blue]{hyperref}

\newtheorem{lemma}{Lemma}
\newtheorem{theorem}[lemma]{Theorem}

\let\srg\relax
\DeclareMathOperator{\srg}{srg}

\title{A Spanning Subgraph of Conway's 99-graph}
\author{Joseph Natal, Oleksiy Al-saadi}

\begin{document}
\date{}
\maketitle

\begin{abstract}

We show that in any hypothetical $G$ = srg(99,14,1,2) a spanning subgraph with $189$ edges is forced. Moreover, the graph formed out of the remaining edges has the spectrum $(12^1,-4^{30}, 3^{40},-2^6,0^7)$ expressed in Brouwer's notation.

\end{abstract}


\section{Preliminaries} \label{sec:prelims}

\emph{Strongly regular graphs} have a history dating back to the mid-1900s and live at the intersection of statistics, Euclidean geometry, group theory, finite geometry, and extremal combinatorics. The state of this field is summarized in a comprehensive compendium by Brouwer \cite{Brouwer_Van_Maldeghem_2022}. A basic overview follows.

A strongly regular graph $G = (V,E)$ can be described using four arguments: $v = |V|$, $k$ is the degree of any vertex, $\lambda$ is the number of shared neighbors of any two adjacent vertices, and $\mu$ is the number of shared neighbors of any two non-adjacent vertices. By $\srg(v,k,\lambda,\mu)$ we denote the class of strongly regular graphs having these properties. Many parameter sets $(v,k,\lambda,\mu)$ give constraints that are impossible to satisfy, and so the corresponding graph does not exist. Different parameters restrict the local structure of a graph. For example, if $\lambda = 0$ then $G$ is triangle-free. Feasible parameters and the status of their existence are recorded in Brouwer's online tables \cite{Brouwer_web}. The adjacency matrix $A(G)$ satisfies the following equations

\begin{equation} \label{eq:srg1}
    A^2 + (\mu-\lambda)A + (\mu - k) I - \mu J = 0
\end{equation}
\begin{equation} \label{eq:srg2}
    A J  = k J
\end{equation}
where $I$ and $J$ are the identity matrix and the all-ones matrix respectively.



\clearpage
\section{Result}
Let $G$ be any srg(99,14,1,2), if it exists. We iteratively add edges to a pivot vertex $f\in V(G)$ until the degree $k$ condition is satisfied. This is done without loss of generality---each new chosen vertex can be chosen arbitrarily. Each of the vertices adjacent to $f$ must have exactly one common neighbor with $f$, so that the following induced subgraph exists 

\begin{figure}[!h]
    \centering
    \includegraphics[width=0.3\linewidth]{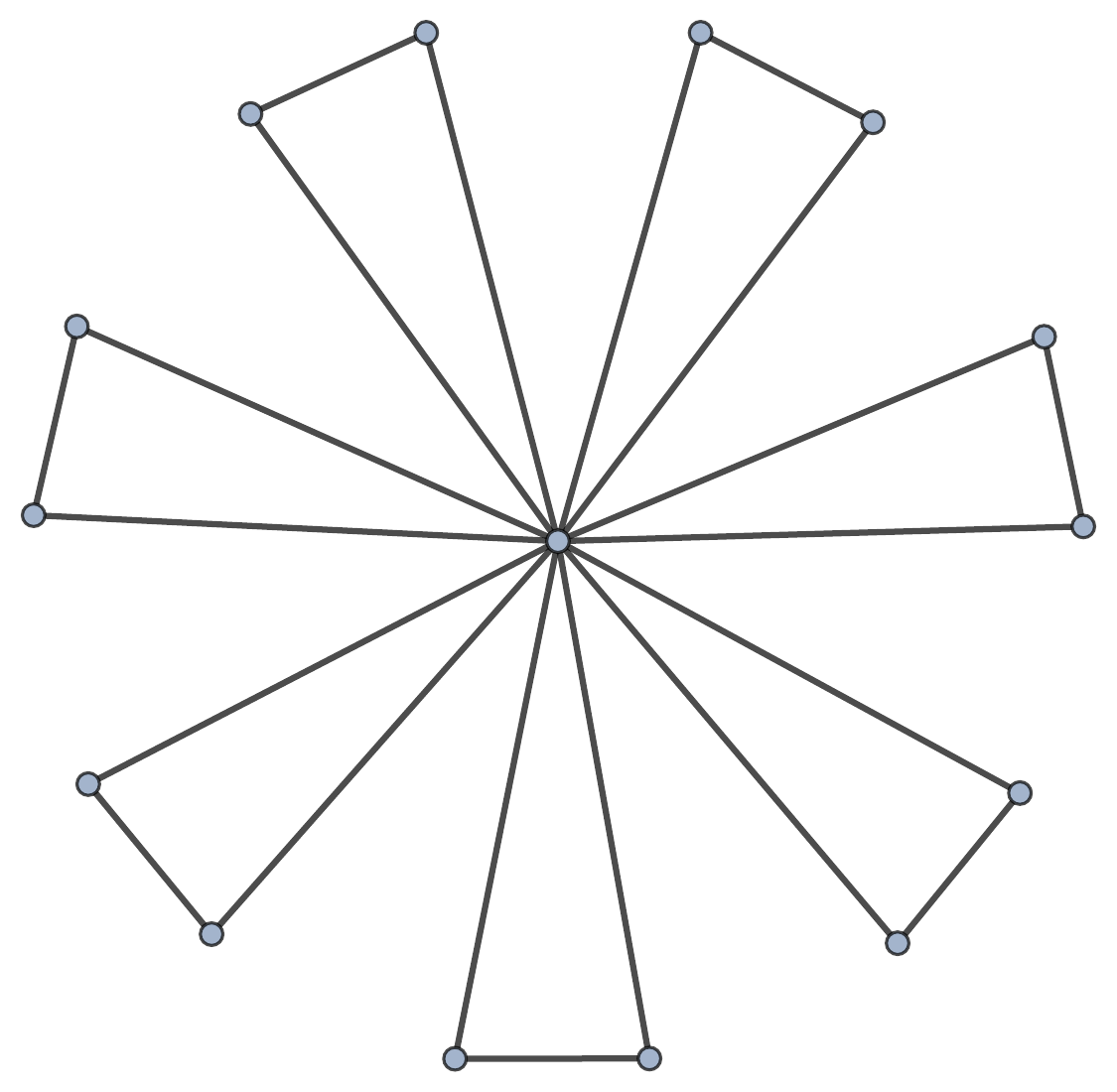}
    \caption{The induced subgraph $F$.}
    \label{fig:placeholder}
\end{figure}

\noindent which we call $F$. We call $F$ a\textit{ 7-bladed fan}, where a \textit{blade} denotes a $K_3$ in $F$. Let the set of remaining vertices be $S = G\setminus V(F)$. Then $|S| = |V| - k - 1 = 84$ vertices remain outside $F$, and all must share exactly two common neighbors with $f$. Let $v \in S$. Note that $v$ cannot be adjacent to two vertices of the same blade of $F$, otherwise the constraint $\lambda=1$ is contradicted. Thus, there are $\binom{k}{2} - \frac{k}{2} = 91 - 7 = 84$ unique and valid pairs of vertices in $F$ that $v$ can be adjacent to. Let $u\in S \setminus v$. Observe that $u$ cannot be adjacent to the same pair of vertices in $F$ that $v$ is adjacent to, otherwise $\mu=2$ is contradicted. Thus the map sending each \(v\in S\) to its two neighbors in \(F\) is injective into the set of \(84\) valid pairs. Since \(|S|=84\), the map is bijective. This implies that the following subgraph exists

\begin{figure}[!h]
    \centering
    \includegraphics[width=0.4\linewidth]{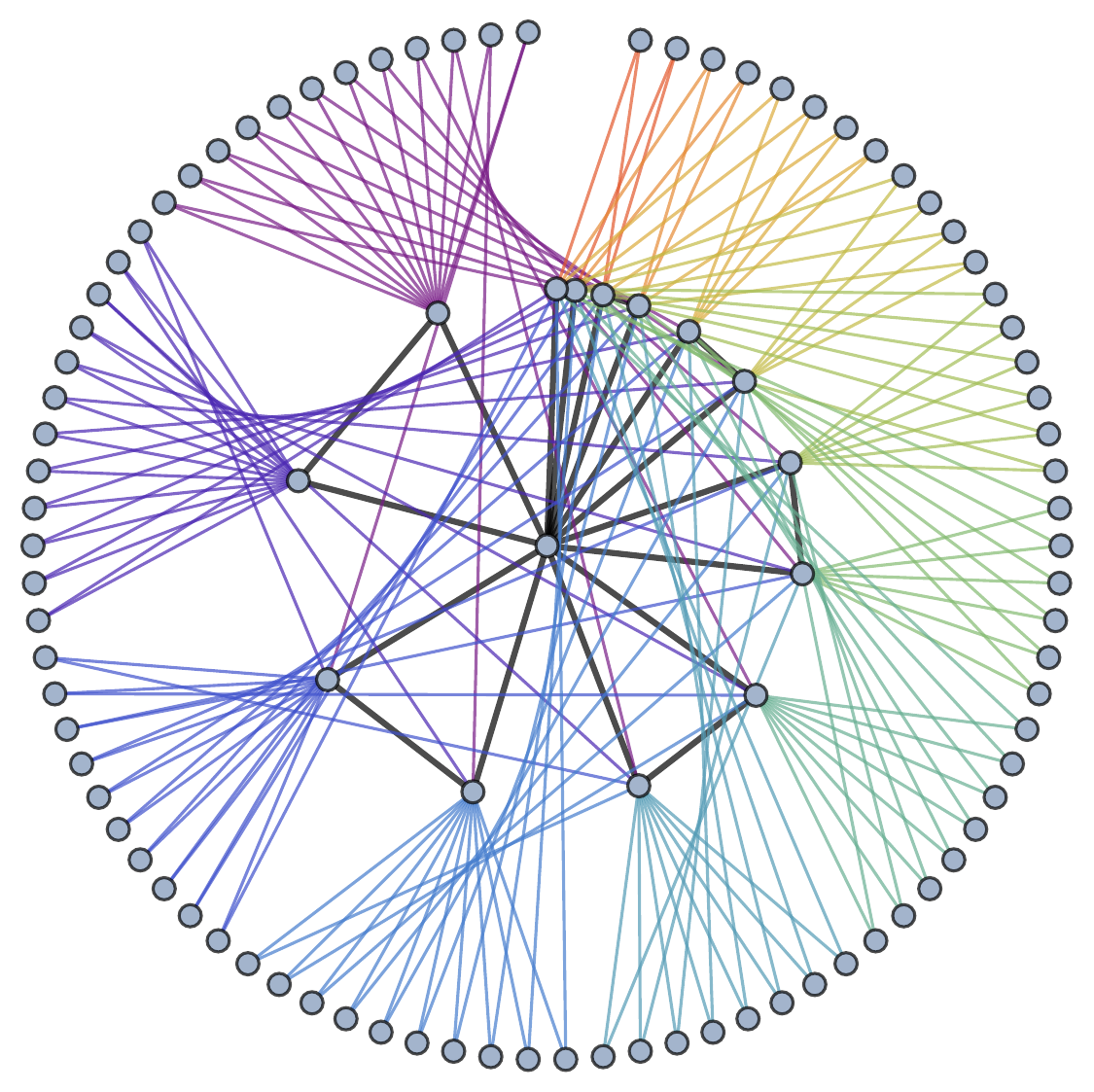}
    \caption{The spanning subgraph $H_1$.}
    \label{fig:enter-label}
\end{figure}

\clearpage

Let the adjacency matrix of $G$ be given by $A$, that of $H_1$ be given by $B$, and that of $H_2=G[S]$ be given by $C$. With the spanning subgraph, the adjacency matrices can be split into blocks

\begin{align*}
B=\begin{pmatrix}
\begin{array}{c|c}
B_0 & B_1^T\\
\hline
B_1 & 0
\end{array}
\end{pmatrix}
\end{align*}

\begin{align*}
A=\begin{pmatrix}
\begin{array}{c|c}
B_0 & B_1^T\\
\hline
B_1 & C
\end{array}
\end{pmatrix}\quad
A^2=\begin{pmatrix}
\begin{array}{c|c}
B_0^2 + B_1^T B_1 & B_0 B_1^T + B_1^T C\\
\hline
B_1 B_0 + C B_1& B_1 B_1^T + C^2
\end{array}
\end{pmatrix}  
\end{align*}
where the dimensions of $B_0,B_1$, and $C$ are respectively  $(k+1) \times (k+1)$, $|S| \times (k+1)$ and $|S| \times |S|$. The blocks $B_0$ and $B_1$ are completely known, representing every vertex in $H_1$ with degree $14$. Inserting these block matrices into Equations \ref{eq:srg1} and \ref{eq:srg2}, the lower right quadrant satisfies
\[
C^2 + (\mu - \lambda)C + (B_1 B_1^T + (\mu - k) I - \mu J) = 0 \]
and the lower left
\[
C B_1  = \lambda B_1 + \mu(J - B_1) - B_1 B_0
\]
and simplifying both equations
\begin{align} \label{eq:first}
C^2 + C + \Delta_1 = 0
\end{align}
\begin{align} \label{eq:second}
CB_1 = \Delta_2
\end{align}
This system is smaller and contains more linear constraints than the system given by Equations \ref{eq:srg1} and \ref{eq:srg2}. We speculate that this system will reduce computer-assisted search times for possible srg(99,14,1,2).

\begin{lemma} \label{lem:eigs}
    The eigenvalues of $C$ are $(12^1,\{3^f,-4^g\},\{1^h,-2^i\},\{0^j,-1^k\})$ where $f,g, \hdots, k$ are multiplicities.
\end{lemma}

\begin{proof}
From \ref{eq:second}, it is clear that \(\Delta_1\) is a polynomial in \(C\), so \(C\) and \(\Delta_1\) commute and are simultaneously diagonalizable because both are real symmetric.
If \(x\) is an eigenvector of \(C\) with eigenvalue \(\theta\), then
\[ \Delta_1x=-(\theta^2+\theta)x. \]
Thus if \(\ell\) is the corresponding eigenvalue of \(\Delta_1\),
\[ \theta^2+\theta+\ell=0\] 
giving the solution $\theta = \frac{1}{2}(-1 \pm \sqrt{1 - 4 \ell} )$. The eigenvalues of $\Delta_1$ are easily resolvable with a computer algebra system. They are, in the canonical notation where the algebraic multiplicity is in the exponent, $(-156^1,-12^{70},-2^6,0^7)$. Solving the quadratic equation for the eigenvalues of $C$ gives \[(12^1,\{3^f,-4^g\},\{1^h,-2^i\},\{0^j,-1^k\})\]
where the eigenvalue of $12$ is forced because the graph is 12-regular.
\end{proof}
\begin{theorem}
    The multiplicities that satisfy the expression in the statement of Lemma \ref{lem:eigs} are $(f,g,h,i,j,k) = (40,30,0,6,7,0)$, yielding the spectrum $(12^1,-4^{30}, 3^{40},-2^6,0^7)$ for $H_2$.
\end{theorem}

\begin{proof}
    
The remaining multiplicities must satisfy the following linear system

\begin{gather}
f + g = 70 \nonumber\\
h + i = 6 \nonumber \\
j + k = 7 \nonumber \\
1 + f + g + h + i + j + k = 84 \nonumber\\
\text{Tr}(C) = 12 + 3 f - 4 g + h - 2 i - k = 0 \nonumber\\
\text{Tr}(C^2) = \text{Tr}(-\Delta_1) = 144 + 9 f + 16 g + h + 4 i + k = 1008 \nonumber
\end{gather}

The first four equations follow from the multiplicities of the eigenspaces of $\Delta_1$. After excluding solutions with negative integers, 8 integer solutions remain to the system. Only one has 140 triangles, calculated from the spectrum using $\#K_3 = \text{Tr}(C^3)/6$. The graph $H_2$ must have 140 triangles as proven by Cesarz et al. \cite{cesarz2023automorphismgroupputativeconway} (Lemma 2.4), and a simple argument why this is true follows: every $v\in S$ forms two separate triangles out of the two vertices in $F$ that $v$ is adjacent to. These triangles are not contained in $H_2$. That leaves the remaining 5 blades of a 7-bladed fan with $v$ as the root entirely in $H_2$. Now $|S|=84$ and $\frac{84\cdot 5}{3} = 140$. The solution spectrum with 140 triangles is $(12^1,-4^{30}, 3^{40},-2^6,0^7)$.
\end{proof}

\bibliographystyle{elsarticle-num}
\bibliography{sample}

\end{document}